\documentclass[9pt,a4paper,onecolumn,oneside]{article}
\usepackage[english]{babel}

\usepackage{enumitem}
\usepackage{lineno}
\usepackage{cite}
\usepackage{amssymb}
\usepackage{amsthm}
\usepackage{amsmath}
\usepackage{color}

\newcommand{\R}{\mathbb{R}}

\newtheorem{definition}{Definition}[section]

\newtheorem{theorem}[definition]{Theorem}
\newtheorem{lemma}[definition]{Lemma}
\newtheorem{corollary}[definition]{Corollary}
\newtheorem{example}[definition]{Example}

\let\epsilon\varepsilon

\begin{document}

	\title{ Weak compactness of sublevel sets in complete locally convex spaces} 
	\author{Pedro P\'erez-Aros \thanks{e-mail:
			pedro.perez@uoh.cl} 
		\\
		\scriptsize{}\\
		\scriptsize{Instituto de Ciencias de la Ingenier\'ia}\\ \scriptsize{Universidad de O’Higgins, Libertador Bernardo O'Higgins 611, Rancagua, Chile}
			\\
		\scriptsize{}\\
	Lionel Thibault	\thanks{e-mail: lionel.thibault@umontpellier.fr}
		\\
	\scriptsize{}\\
\scriptsize{Institut Montpelli\'erain Alexander Grothendieck} \\
\scriptsize{Universit\'e de Montpellier, 34095 Montpellier CEDEX 5, France}\\
\scriptsize{Centro de Modelamiento Matematico}\\ \scriptsize{Universidad de Chile}
}

\maketitle
\begin{abstract}
	In this work we prove that if $X$ is a complete locally convex space and $f:X\to \mathbb{R}\cup \{+\infty \}$ is a function such that $f-x^\ast$ attains its minimum for every $x^\ast \in U$, where $U$ is an open set  with respect to the  Mackey topology in $X^\ast$, then for every $\gamma  \in \mathbb{R}$ and  $x^\ast \in U$  the set $\{ x\in X : f(x)- \langle x^\ast , x \rangle \leq \gamma	\}$ is relatively weakly compact. This result corresponds to an extension of Theorem 2.4 in [J. Saint Raymond, Mediterr. J. Math. 10 (2013), no. 2, 927--940].  
 Directional James compactness theorems are also derived.
	
\end{abstract}

\section{Introduction} 
  We recall first the following well-known J.J. Moreau's result {(see \cite[\S8.f, p. 49]{Moreau}; or see \cite[Theorem I-14]{CastaingValadier}, \cite[Theorem 6.3.9]{Laurent}) } that we state as a theorem: 
	
\begin{theorem}\label{theorem0} 
Let $X$ be a (Hausdorff) locally convex space and $f: X\to\mathbb{R}\cup\{+\infty\}$ be a proper lower semicontinuous convex function. Given $x^*\in X^*$, all the sublevel sets  
$\textnormal{lev}_{f-x^\ast}^{\leq }(\gamma):=\{ x  \in X :  f (x)-\langle x^\ast , x\rangle \leq  \gamma \}$ are weakly compact in $X$ for all 
$\gamma\in\mathbb{R}$ if and only if the conjugate function $f^*$ of $f$ is finite at $x^*$ and 
continuous at $x^*$ with respect to the Mackey topology on $X^*$. 
\end{theorem}	
	
	Conjugate functions and Mackey topology are recalled in the next section. In addition to that theorem, J. Saint Raymond provided recently in \cite[Theorem 2.4]{MR3045687}  another condition for weak compactness of sublevels in the Banach setting: 
	
\begin{theorem}\label{theorem1}
	Let $X$ be a Banach space and let $f: X\to \mathbb{R}\cup\{+\infty\}$ be a weakly lower 
	semicontinuous function. If $f-x^\ast$ attains its minimum for every $x^\ast \in  X^\ast$, then the
	sublevels $\textnormal{lev}_f^{\leq }(\gamma)=\{ x  \in X :  f (x) \leq  \gamma \}$ are weakly compact.
\end{theorem}

This result can be understood as a functional counterpart of James' Theorem (\cite{MR0170192,MR0165344}). Posteriorly, W. B. Moors \cite{MR3652791} provided a shorter proof of Theorem \ref{theorem1}. 

In this work we extend Theorem \ref{theorem1} to complete locally convex spaces by using ideas  in \cite{MR3652791} and arguments from  convex analysis.  We also  make the link between the class of  epi-pointed functions and the variational property that $f-x^\ast$ attains its minimum for every $x^\ast \in U$, where $U$ is an open set with respect to the  Mackey topology. Directional James weak compactness 
theorems for sets are also derived. 

\section{Notation and preliminaries\label{Section2}}

Throughout   the paper   $X$ will be a {\em (real) complete  (Hausdorff) locally convex space  for some topology $\theta$}  and $X^{\ast }$ its topological dual. The canonical bilinear  form $\langle \cdot ,\cdot
\rangle :X^{\ast }\times X\rightarrow \mathbb{R}$ is given by $\langle x^{\ast
},x\rangle :=x^{\ast }(x)$. It will be convenient to denote $\R_+:=[0,+\infty[$. 

For a set $A\subseteq X$ (or $X^{\ast }$), we denote by $\textnormal{Int}(A)$, $%
\overline{A}$, $\textnormal{co}(A)$ and $\overline{\textnormal{co}}(A)$, the interior, the closure, the \emph{%
	convex hull} and  the \emph{closed convex hull} of $A$, respectively. The \emph{indicator}
and the \emph{support} functions of $A$ are,
respectively, $$\delta _{A}(x):=%
\begin{cases}
0\qquad & x\in A \\ 
+\infty & x\notin A,%
\end{cases} \text{ and } \sigma _{A}(\cdot):=\sup\limits_{x\in  A }\langle \cdot ,x \rangle. $$

    The \emph{polar} and \emph{strict polar} of $A$ are the sets 
\begin{align*}
A^{o}&:=\{x^{\ast }\in X^{\ast }\mid \sigma_{A}(x^\ast) \leq 1 \},\\
A^{\circ}_s&:=\{x^{\ast }\in X^{\ast }\mid \sigma_{A}(x^\ast) < 1\},
\end{align*}
 respectively.  We define the \emph{directional asymptotic cone} of $A$  as 
 \begin{equation*}
 { A_{\infty}:=\{ y \in X : \exists x\in A,\,\forall \lambda\geq 0,\,x + \lambda y \in   A   \}. }  
 \end{equation*}

The weak topology in $X$ and the weak star topology in $X^{\ast }$ are denoted by $w$ and $w^{\ast }$ 
respectively. In $X^\ast$  the  {\em Mackey topology}  is 
denoted by $\tau (X^{\ast },X)$. We recall that  $\tau (X^{\ast },X)$ is the topology on $X^\ast$ whose a basis of \mbox{$\tau(X^*,X)$-neighborhoods} of zero is constituted by the  polar sets of  all weakly compact  circled (balanced) convex sets in $X$ (see, e.g., \cite{MR1741419}).

Given a function $f:X\rightarrow \mathbb{R}\cup \{+\infty \}$, the
(effective) \emph{domain} of $f$ is the set $\textnormal{dom} f:=\{x\in X : 
f(x)<+\infty \}$. We say that $f$ is \emph{proper} if $\textnormal{dom} f\neq \emptyset $, and \mbox{$\tau$-\emph{inf-compact}} if for every $\lambda \in \mathbb{R}
$ the sublevel set $\{x\in X\mid f(x)\leq \lambda \}$ is $\tau$-compact. The \emph{%
	conjugate} of a proper function\ $f$ is  $f^{\ast }:X^{\ast
}\rightarrow \mathbb{R}\cup \{+\infty \}$ defined by 
$$
   f^{\ast }(x^{\ast }):=\sup_{x\in X}\big(\langle x^{\ast },x\rangle -f(x)\big),
$$
and the \emph{closed convex  hull} of $f$ is the function $\overline{\textnormal{co}} f$ such that 
$$
   \overline{\textnormal{co}} \big(  \textnormal{Epi}\,f \big)=\textnormal{Epi}\,\big( \overline{\textnormal{co}} f \big) ,  
$$
where $\textnormal{Epi}\,f:=\{(x,\alpha)\in X\times\mathbb{R}: f(x)\leq \alpha\}$ is the epigraph of $f$. 

  The {\em Moreau-Rockafellar subdifferential} is defined at $x\in \textnormal{dom} f$ as 

$$
    \partial f(x):= \{x^* \in X^\ast: \langle x^*,y-x\rangle +f(x)\leq f(y), \;\forall y\in X\}, 
 $$ 
hence 
\begin{equation}\label{eq-DefMRSubd}
\partial f(x)= \{ x^* \in X^* : f(x) + f^\ast(x^\ast) = \langle x^\ast 	, x\rangle	\}. 
\end{equation} 
When $f$ is not finite at $x$, one sets $\partial f(x)=\emptyset$. 
The inverse set-valued mapping $(\partial f)^{-1}$ of $\partial f$  is then the set-valued mapping from 
$X^*$  into $X$ defined for all $x^{\ast}\in X^*$ by 
\begin{align*}
\left( \partial f	\right)^{-1} (x^\ast)=\{ x \in X : x^\ast \in \partial f(x)		\}, 
\end{align*} 
and we note that 
\begin{equation}\label{eq-SubdFctConj}
     x^{\ast}\in \partial f(x) \, \Rightarrow \, x\in \partial f^{\ast}(x^\ast) . 
\end{equation}   
When $f$ is convex and finite at $x$, its directional derivative in any direction $h\in X$ 
$$
    f'(x;h)=\lim_{s\to 0^+}\frac{f(x+sh)-f(x)}{s}=\inf_{s>0}\frac{f(x+sh)-f(x)}{s} 
$$
exists in $[-\infty,+\infty]$, so noting that $f'(x;0)=0$ it follows (as well-known) that 
\begin{equation}\label{eq-SubdDirDer}
    \partial f(x)=\partial (f'(x;\cdot))(0). 
\end{equation}  
If the function $f:X\to\mathbb{R}\cup\{+\infty\}$ is proper, lower semicontinuous and convex,   one also knows that 
\begin{equation}\label{eq-SubdInvConj}
     (\partial f)^{-1}(x^{\ast})=\partial f^\ast(x^{\ast}) \quad\text{for all}\; x^{\ast}\in X^{\ast}, 
\end{equation}
where $\partial f^{\ast}(x^{\ast})$ is taken for $X^{\ast}$ endowed with the weak$^{\ast}$ topology, so its topological dual is $X$.    

 The \emph{inf-convolution} of $f$ with another function $g:X\rightarrow \mathbb{R}\cup \{+\infty \}$ is the function 
$$ 
    (f\square g)(\cdot):=\inf\limits_{z\in X}\{f(z)+g(\cdot -z)\}.   
$$

     Finally, we present the definition of an epi-pointed function in locally convex spaces.  This class of functions has been applied to obtain diverse extensions of classical results, which were known exclusively  for functions defined in Banach spaces, to arbitrary locally convex spaces   (for
more details about these  results, see \cite{MR3507100,Correa2017,Correa20182,Correa20183,Correa2018,Perez1}
and the references therein). As far as we know, this class of functions was first introduced 
in finite dimensions in \cite{MR1416513}, and the extension to locally convex spaces was posteriorly introduced in \cite{MR3509670} with the name of \emph{Mackey epi-pointed functions}. 

\begin{definition}
	\label{definitionepipointed} The function $f:X\rightarrow \mathbb{R}\cup
	\{+\infty \}$ is said to be epi-pointed if $f^{\ast }$ is proper and $\tau
	(X^{\ast },X)$-continuous at some point of its domain.
\end{definition}

\section{Main results} 
 In this section we will establish the extension of Saint Raymond' Theorem to the context of the complete locally 
convex space $X$. This will be obtained as a consequence of Theorem \ref{mainresult} which can be seen as 
the basic theorem. We will also derive from Theorem \ref{mainresult} diverse consequences with directional James conditions for weak compactness of subsets in $X$. 

    The statement of Theorem \ref{mainresult} will involve the following class of functions. 

\begin{definition}\label{definition3.1}
	Consider a set $K\subseteq X^\ast$ with $0\in K$.  We define $\mathcal{E}(K)$ as the class of all  functions  $\varphi : X\to \mathbb{R}\cup \{+\infty \}$ such that
	\begin{enumerate}[label={(\roman*)},ref={(\roman*)}]
		\item  $\varphi^*  : X^\ast \to \mathbb{R} \cup \{ +\infty \}$ with $\varphi^\ast (0)=0$; 
		\item $ K_{\infty} \subseteq \textnormal{dom}\,\varphi^\ast \subseteq K$; 
		\item  $\sup\{ \varphi^\ast (\eta x^\ast) :  \eta >0,\; \eta x^\ast \in \textnormal{dom }\,\varphi^\ast \}= +\infty$, for all $x^*\in  K \backslash K_\infty $; 
		\item $\left(  \partial \varphi\right)^{-1}(x^\ast)\neq \emptyset$ for all   
		$x^\ast \in \textnormal{dom}\,\varphi^\ast $.
	\end{enumerate}
\end{definition}	
   Concerning the above condition $(iv)$ it is worth noting that given $x^*\in X^*$ one has 
\begin{equation}\label{eq-NemptInvSubd}
(\partial \varphi)^{-1}(x^\ast)\not=\emptyset \;\Leftrightarrow \; 
 \varphi -x^* \;\text{attains its minimum on}\; X. 
\end{equation}   
\vskip 0.3cm 

 The  result of Theorem \ref{mainresult} below corresponds to an abstract extension of Theorem \ref{theorem1}, which relies on the non-emptiness of $\mathcal{E}(K)$. Its proof uses the main ideas of 
W. B. Moors \cite{MR3652791}.

\begin{theorem}\label{mainresult} 
Let $K\subseteq X^\ast$ with $0\in K$, $\mathbb{R}_{+} K =X^*$ and $\mathcal{E}(K)\neq\emptyset$, and 
	let $f: X \to \mathbb{R} \cup \{ +\infty \}$ be a function such that  $(\partial f)^{-1}(x^\ast) \neq \emptyset$ for all $x^\ast \in K$. Then for every $\gamma  \in \mathbb{R}$ and every  $\varphi \in \mathcal{E}(K)$ the sets 
	\begin{align*}
	\textnormal{Ep}_f^{\varphi}(\gamma)&:=\{ 	(x ,\alpha) 	\in X \times  \mathbb{R} : (x,\alpha) \in \textnormal{Epi}\, f +  \textnormal{Epi}\, \varphi	\text{ and } \alpha  \leq \gamma  \},\\
	{S}_{f\square\varphi}(\gamma)&:=\{ (x,\alpha ) \in X \times  \mathbb{R} :  f\square \varphi (x) \leq \alpha \leq \gamma \}, \\
	 S_f(\gamma)	&:=\{ (x,\alpha ) \in X \times  \mathbb{R} :  f(x) \leq \alpha \leq \gamma \},\\
	\textnormal{lev}_f^{\leq }(\gamma)&:=\{ x  \in X :  f (x) \leq  \gamma \}
	\end{align*}
	are relatively weakly compact.  
\end{theorem}
\begin{proof}
	 Fix any $\varphi\in\mathcal{E}(K)$.  By our assumption 
	$(\partial f)^{-1}(0)\neq \emptyset$ we see that $\inf_Xf$ is finite (see \eqref{eq-NemptInvSubd}) and that choosing $x_0\in \partial f^{\ast}(0)\neq \emptyset$ (see \eqref{eq-SubdFctConj})  and setting  
	$\beta_0:=\inf_Xf\in\mathbb{R}$ give $f^{\ast}\geq \langle \cdot,x_0\rangle +\beta_0$ on $X^*$.   Putting $\alpha_0:=1-\inf_Xf$, for the function $g:=f(\cdot+x_0)+\alpha_0$ we have $\inf_Xg=1$ 
	and $g^*=f^*-\langle \cdot,x_0\rangle -\alpha_0\geq \beta_0 -\alpha_0$. 
	Further, it is not difficult to see that 
	$$
	   \textnormal{Ep}_g^{\varphi}(\gamma)=(-x_0,\alpha_0)+\textnormal{Ep}_f^{\varphi}(\gamma-\alpha_0), 
		\; 
		 {S}_{g\square\varphi}(\gamma)=(-x_0,\alpha_0)+{S}_{f\square\varphi}(\gamma-\alpha_0)
	$$
	$$
	   S_g(\gamma)=(-x_0,\alpha_0)+S_f(\gamma-\alpha_0),\; 
		 \textnormal{lev}_g^{\leq }(\gamma)= -x_0+ \textnormal{lev}_f^{\leq }(\gamma-\alpha_0). 
	$$
	Given this, we may and do suppose that $\inf_Xf=1$ and $\inf_{X^{\ast}}f^*>-\infty$.  Then we consider $E:= \textnormal{Epi}f + \textnormal{Epi}\varphi $ and (since $\inf_X\varphi=0$ by Definition \ref{definition3.1}) we can define
	$T : E \to X \times \mathbb{R}$ by $T(x,\alpha)=\alpha^{-1} (x,-1)$. Straightforwardly,  $T$ is a homeomorphism from $E $ to $T(E)$ relative to the topologies induced by the weak topology.    
	
	 Let us prove that $A:= T(E) \cup \{ (0,0) \}$ is relatively weakly compact. Fix any arbitrary 
	$(x^\ast,r ) \in K \times \mathbb{R}$ and consider two cases.

	\textbf{Case 1:}  $x^\ast \in K_\infty $ and $f^*(\eta x^\ast) + \varphi^*(\eta x^\ast)\leq  \eta r$ for all $\eta >0$.
	
	  In this case for all $\eta >0$, and for any $a,b \in X$
	$$
	  \langle x^\ast , a \rangle  - \eta^{-1}f(a) +\langle x^\ast , b  \rangle - \eta^{-1}\varphi(b)
		 \leq r. 
	$$
	Consequently, for every $(x,\nu)=(a+b,\alpha +\beta ) \in E$ 
	with $(a,\alpha)\in \textnormal{Epi}f$ and $(b,\beta)\in\textnormal{Epi}\varphi$ and for every $\eta>0$ 
	\begin{align*}
	\langle (x^\ast, r) , T(x,\nu)\rangle  &=(\alpha +\beta )^{-1}\big(  \langle x^\ast  ,a \rangle + \langle x^\ast  ,b \rangle - r \big)\\ 
	&\leq
	(\alpha +\beta )^{-1}\big(  \eta^{-1}f(a) +  \eta^{-1}\varphi(b)  \big)
	\leq \eta^{-1}.
	\end{align*} 
	Since $\eta >0$ is arbitrary, we get $\langle (x^\ast, r) , T(x,\nu)\rangle   \leq 0$, 
 and additionally 	
$$ 
    \langle (x^\ast, r) , (0,0)\rangle  = 0,
$$ 
so we obtain that $(x^\ast, r)$ attains its maximum  over $A$ at $(0,0)$.
	\newline
	\textbf{Case 2:} Either $(i)$ $x^\ast \in K_\infty$ and $f^*(\eta x^\ast) +\varphi^*(\eta x^\ast) >  \eta r$ for some $\eta >0$, or $(ii)$  $x^* \in K \backslash K_\infty$.
	
In  either $(i)$ or $(ii)$ we consider the proper lower semicontinuous convex function $h$ on $\R$ defined by  
$h(t)= f^*(t x^\ast) + \varphi^{\ast}(t x^\ast) - rt$; this is a continuous function over its effective domain 
{(see, for example, \cite[Theorem 10.2]{MR1451876}).} Moreover, in both cases 
$h(0) =f^*(0) + \varphi^\ast (0) = -1  < 0$ and 
$\sup\limits_{ \substack{t \in \textnormal{dom } h\\ t>0}} h(t) >0$. Indeed, in (i) the latter inequality is obvious, and in (ii)  we have that 
\begin{align*}
   {\sup\limits_{\substack{ t \in \textnormal{dom } h\\ t>0}} h(t) } \geq 
\inf_{X^\ast} f^\ast  +  
\sup\limits_{ \substack{ t x^\ast \in \textnormal{dom } \varphi ^\ast \\ t>0 }} 
\varphi ^\ast(t x^\ast) -t_0\max\{0,r\}=+\infty, 
\end{align*}
 where $t_0 :=\sup\{ t >0 :   t x^\ast \in  \textnormal{dom } \varphi ^\ast   \} <+\infty$ because 
$x^* \not\in K_\infty$ and $0\in K$.  Then,  by \emph{Bolzano's Theorem}  there exists  
{a real $\delta >0$ such that }  
	$$  
	    f^*(\delta x^\ast) + \varphi^\ast (\delta x^\ast) =\delta r.
	$$ 
	Because   $\delta x^{\ast}\in \textnormal{dom }\ \varphi^\ast$ and 
	$\delta x^{\ast} \in K$ (since $\textnormal{dom}\,\varphi^{\ast}\subset K$), 
	 there exist by \eqref{eq-DefMRSubd} and the hypothesis relative to  $\varphi$ and $f$  some 
	$a_0,b_0 \in X $ such that  
	$\varphi^\ast(\delta x^*) = \langle \delta x^\ast , b_0 \rangle  - \varphi(b_0)$ and 
	$f^\ast(\delta x^*) = \langle \delta x^\ast , a_0 \rangle  - f(a_0)$,   
	 and this implies that
$$ \langle x^\ast , a_0 \rangle  +\langle x^\ast , b_0 \rangle   - \delta^{-1}(f(a_0) + \varphi(b_0)) = r,$$
and for all $a,b \in X$
		$$\langle x^\ast , a \rangle  +\langle x^\ast , b \rangle  - \delta^{-1} (f(a) + \varphi(b)) \leq r.$$

	  Then, for every $(x,\nu)=(a+b,\alpha +\beta ) \in E$ with some $(a,\alpha)\in\textnormal{Epi}f$ and 
		$(b,\beta)\in \textnormal{Epi}\varphi$,  we have 
	\begin{align*}
	\langle (x^\ast, r) , T(x,\nu)\rangle  &=(\alpha +\beta )^{-1}\big(  \langle x^\ast  ,a \rangle + \langle x^\ast  ,b \rangle - r \big)\\ 
	&\leq
	(\alpha +\beta )^{-1}\big(  \delta^{-1}f(a) +  \delta^{-1}\varphi(b)  \big)
	\leq \delta^{-1},
	\end{align*} 
  $\langle (x^\ast, r) , T(a_0+ b_0,f(a_0) + \varphi(b_0))\rangle  = \delta^{-1}$ and 	$\langle (x^\ast, r) , (0,0)\rangle=0$,
 and all together entail  that
 $(x^\ast, r)$ attains its maximum over $A$ at $T(a_0+ b_0,f(a_0) + \varphi(b_0))$. 
	
	Since {$\mathbb{R}_+( K \times \mathbb{R})=X^*\times \mathbb{R}$,} we conclude that     every $(x^\ast, r)\in X^*\times \mathbb{R}$ attains its maximum over $A$. Then, since the locally convex space $X$ is complete,  by James' Theorem (see, e.g., \cite[Theorem 6]{MR0165344}) $A$ is relatively weakly compact. 
	
	   Fix any real $\gamma \geq 1$ and consider any net $(x_i,\nu_i)_{i\in {I}}$  in  
$\textnormal{Ep}_f^{\varphi}(\gamma)$. Since $1\leq \nu_i\leq \gamma$, we may suppose that 
$(\nu_i)_i$ converges to some $\nu\in [1,\gamma]$. Since   
$z_i:=\nu_i^{-1}(x_i,-1)$ is in $T(E)$, there is a subnet $(z_{s(j)})_{j\in {J}}$ converging weakly 
to some $(u,\lambda)$ in $X\times \mathbb{R}$, so $(x_{s(j)})_{j\in {J}}$ converges weakly to 
$\nu{u}$ in $X$. Then $(x_{s(j)},\nu_{s(j)})_{j\in {J}}$ converges weakly, which justifies that 
$\textnormal{Ep}_f^{\varphi}(\gamma)$ is relatively weakly compact. 

            Now noting that $(x,\nu)\in {S}_{f\square\varphi}(\gamma)$ entails that 
$(x,\nu+1)\in \textnormal{Ep}_f^{\varphi}(\gamma +1)$, we see that 
${S}_{f\square\varphi}(\gamma) \subset (0,-1)+ \textnormal{Ep}_f^{\varphi}(\gamma +1)$, hence 
${S}_{f\square\varphi}(\gamma)$ is relatively weakly compact. 

     In order to prove the desired property of $	S_f(\gamma)	$, consider $x_0\in \textnormal{dom} \varphi$ and take the function $\tilde{f}(\cdot) := f(\cdot + x_0)$. Since  $\tilde{f}\square \varphi (x)\leq f(x) +\varphi(x_0)$, we get that 
 $S_f(\gamma)	\subseteq 	{S}_{\tilde{f}\square\varphi}(\gamma+\varphi(x_0))$,  which proves that  $S_f(\gamma)$ is relatively weakly compact.

  Finally, 	the weak compactness of $\textnormal{lev}_f^{\leq }(\gamma)$ follows from the inclusion 
	$\textnormal{lev}_f^{\leq }(\gamma) \times\{\gamma\}\subset S_f(\gamma)$.

\end{proof}  

   We state now in a  lemma a  chain rule for the composition of convex functions. The result is a particular case of a more general version \cite[Corollary 1]{MR1721727}.

\begin{lemma}\label{lem-Comp}
	Let $(Y,\tau)$ be an arbitrary locally convex space and $g: Y \to \mathbb{R}$ be a lower semicontinuous convex  function. Consider  $\varPsi : \mathbb{R} \to \mathbb{R}\cup\{ +\infty\}$ a non-decreasing continuous differentiable function in its domain. Then for every $x \in \textnormal{dom} g$ with $\varPsi'(g(x)) >0$ we have that 
	\begin{align}\label{compositionrule}
		\partial ( \varPsi  \circ g) (x)= \varPsi'( g(x)) \partial g(x).
	\end{align} 
\end{lemma}
\begin{proof}
In order to apply  \cite[Corollary 1]{MR1721727} we must verify  \cite[$(C.Q_4)$, p. 29]{MR1721727}. Consider  $x\in Y$ with $\varPsi'(g(x)) >0$.  Pick an arbitrary neighborhood $V $ of zero in $Y$.   
Taking a real $\lambda \geq 0$ such that $g(x)\in \textnormal{int}(\textnormal{lev}_\Psi^{\leq }(\lambda))$ 
 and $x \in \lambda V$,  we see that 
$$0 \in \textnormal{int}\left(			\textnormal{lev}_\Psi^{\leq }(\lambda) - g( 	\lambda V)		)		\right).$$
Then applying  \cite[Corollary 1]{MR1721727} we have that \eqref{compositionrule} holds.

\end{proof} 

 The next  lemmas provide various conditions for the non-emptiness of $\mathcal{E}(K)$. The first lemma considers the case of a strict sublevel of a finite-valued lower $w^*$-semicontinuous convex funcion on $X^*$.

\begin{lemma}\label{Lemmanuevo}
	Let $g: X^\ast \to \mathbb{R}$ be a weak$^\ast$ lower semicontinuous convex  function such that $0\in \textnormal{lev}_g^{< }(\beta):=\{ x^\ast\in X^{\ast} : g(x^\ast) < \beta		\}$  and
	\begin{equation*}
	 \partial g(x^\ast) \neq \emptyset, \; \forall x^\ast \in  \textnormal{lev}_g^{< }(\beta).
	\end{equation*} 
	Then $\mathcal{E}(\textnormal{lev}_g^{< }(\beta)) \neq \emptyset$; in fact, there exists 
	a proper lower semicontinuous convex function 
	$\varphi\in \mathcal{E}(\textnormal{lev}_g^{< }(\beta))$ such that $\varphi^*$ is bounded from below on $X^*$. 
\end{lemma}

\begin{proof}
	Consider a  non-decreasing convex function $\varPsi:\mathbb{R}\to ]0,+\infty]$ which on $]-\infty,\beta[$ is finite,   injective  and continuously differentiable   with  $\varPsi'(x) > 0$  for all $x\in ]-\infty, \beta[$, and which satisfies  $\lim\limits_{ u \to \beta^{-} } \varPsi(u) =+\infty$; for example one can consider
	 $\varPsi$ defined by $\varPsi(u)=+\infty$ for all $u\in [\beta,+\infty[$ and 
	$$
	   \varPsi(u)= \frac{1}{\beta -u} \quad \text{for all}\; u\in ]-\infty,\beta[ .   
	$$

	     Now defining $\psi (x^\ast) = \varPsi ( g(x^\ast))-\varPsi(g(0))$ for all $x^{\ast}\in X^{\ast}$,  
			it is not difficult  to see 
	that $\psi$ is a proper weak$^\ast$ lower semicontinuous convex function and $\textnormal{dom}\, \psi = \textnormal{lev}_g^{< }(\beta)$.  Then there exists a proper lower semicontinuous convex function $\varphi:X\to\mathbb{R}\cup\{+\infty\}$ such that $\psi=\varphi^{\ast}$. Further, 
	$\textnormal{dom}\,\varphi^{\ast}=\textnormal{dom}\,\psi= \textnormal{lev}_g^{< }(\beta)$ and $\varphi^{\ast}(0)=0$ and $\varphi^{\ast}(X^{\ast})\subset\mathbb{R}\cup\{+\infty\}$.  
	We claim that $\varphi^{\ast}$ satisfies:
		\begin{enumerate}
			\item[(iii)]\label{Lemmanuevoc} $\sup\limits_{\substack{  \eta x^\ast	 
					\in\textnormal{dom} \varphi^{\ast},\; \eta > 0 }  } \varphi^{\ast}( \eta x^\ast) = +\infty$ for all $x^\ast \in   \textnormal{lev}_g^{< }(\beta) \backslash \left( \textnormal{lev}_g^{< }(\beta) \right)_{\infty}$.
		\item[(iv)] \label{Lemmanuevoa} $(\partial \varphi )^{-1}(x^\ast)\neq \emptyset $ for all 
	$x^\ast \in \textnormal{dom} \varphi^{\ast}$. 
		
	\end{enumerate}
Indeed, pick $x^\ast \in \textnormal{lev}_g^{< }(\beta) \backslash \left( \textnormal{lev}_g^{< }(\beta) \right)_{\infty}$. Then $0\in \textnormal{lev}_g^{< }(\beta)$ and 
$x^{\ast}\not\in \left( \textnormal{lev}_g^{< }(\beta) \right)_{\infty}$, so by the convexity  of 
$ \textnormal{lev}_g^{< }(\beta) $ we have that, for some $\lambda >0$, 
$g(\lambda x^\ast) \geq \beta$. Therefore, by the continuity on $\R$  of the function 
$s\mapsto  g(s x^\ast)$, there is some   $\eta_0 >0$ such that $g(\eta_0x^{\ast})=\beta$. Putting 
$\eta_1:=\inf\{\eta>0:g(\eta{x^\ast})=\beta\}$, we see by the continuity of $g$ that $\eta_1$ is a real number with $0<\eta_1\leq \eta_0$  such that $g({\eta_1}x^\ast)=\beta$ and 
$\lim_{\eta\to \eta^-_1}g({\eta}x^\ast)=\beta$. Then $\eta{x^\ast}\in \textnormal{dom}\psi$ for every 
$\eta\in ]0,\eta_1[$ and $\lim_{\eta\to \eta^-_1}\psi({\eta}x^\ast)=+\infty$
This ensures that  (iii) holds.  

   By the chain-rule in Lemma \ref{lem-Comp} we also have that 
 \begin{align*}
 \partial {\varphi}^{\ast} (x^\ast)  = \varPsi'(g(x^\ast)) \partial g(x^\ast),  
 \text{ for all } x^\ast  \text{ with } \varPsi'(g(x^\ast)) >0,
\end{align*} 
then (iv) holds by \eqref{eq-SubdInvConj}, so   $\mathcal{E}(\textnormal{lev}_f^{< }(\beta)) \neq \emptyset$.  
Finally, since $\varPsi \geq 0$ we have that $\varphi^*=\psi$ is bounded from below. 
\end{proof}

  The second lemma is related to the polar in $X^*$ of a circled weakly compact set. 
			
\begin{lemma}\label{Key:lemma}
	Let $K$ be a nonempty  circled convex subset in $X$ with $\mathbb{R}_{+} K^\circ=X^\ast$. Then 	$\mathcal{E}(K^\circ)\neq \emptyset$  if and only if $K$ is relatively weakly compact; in such a case there exists a proper weakly inf-compact convex function $\varphi\in \mathcal{E}(K^\circ)$ 
	with $\varphi ^* \geq 0$. 
\end{lemma}
\begin{proof}
	First suppose that $\mathcal{E}(K^\circ)\neq \emptyset$. Then, the function $f:=\sigma_{K^\circ}$ satisfies the properties:  $0\in (\partial f)^{-1}(x^\ast)$ for every $x^\ast \in K^\circ$, and   $\mathbb{R}_{+} K^\circ=X^\ast$ (by our assumptions). Then, by Theorem \ref{mainresult} we have that $\textnormal{lev}_f^{\leq }(1)$ is  relatively weakly compact, and since $K\subseteq \textnormal{lev}_f^{\leq }(1)$ we get the desired property for $K$. 
	
	Now suppose that $K$ is relatively weakly compact.  Consider 
	$\zeta:\R\to\R\cup\{+\infty\}$ with $\zeta(t)=\tan(t)$ for $t\in [0,\pi/2[$ and $\zeta(t)=+\infty$ 
	otherwise, and define $h:X^{\ast}\to \R\cup\{+\infty\}$ by 
	$$
	   h(x^*):= \zeta \big(\sigma_{K}(x^\ast) \frac{\pi}{2}\big) + \delta_{K^\circ_s}(x^\ast) 
		\quad \text{for all}\; x^{\ast}\in X^{\ast}.
	$$
	We notice that  $h$ is convex 
	and $\tau(X^\ast,X)$-continuous on $K^{\circ}_s$
	(and so weak$^*$ lower semicontinuous on $X^{\ast}$). Then, we define 
	$\varphi:X\to \R\cup\{+\infty\}$ by 
	$$
	 \varphi (x)=h^\ast (x)=\sup_{x^\ast \in X^\ast} \big( \langle x^\ast , x \rangle - h(x^\ast)	\big) 
	 \quad\text{for all}\; x \in X, 
	$$   
	and it is well-known that its conjugate function satisfies   $\varphi^{*}=h$ 
	(see, e.g., \cite[Theorem I.4]{CastaingValadier}, or \cite[Theorem 6.3.7]{Laurent}). Moreover, $\varphi^{\ast}(0)=0$ and: 
	\begin{enumerate}[label={(\alph*)},ref={(\alph*)}]
	
		\item We have  
		\begin{align*}
		\left( K^\circ\right)_{\infty} & =\{ x^\ast \in  X^\ast : \sigma_{K}(x^\ast) \leq 0 \} 
		    = \{x^{\ast}\in X^{\ast}: \sigma_K(x^{\ast})=0\} \\
		 & \subseteq  \textnormal{dom}\,\varphi^\ast = \{ x^* \in X^* : \sigma_{K}(x^*) <1\}=K_s^\circ; 
		\end{align*}
		\item  For any $x^{\ast}\in K^\circ \backslash \left( K^{\circ}\right)_{\infty} $, or equivalently $0 <\sigma_K(x^\ast)\leq 1$,  since $\tan(t)\to +\infty$ as $t\to (\pi/2)^-$ we see that $$\sup\limits_{\substack{ \eta x^\ast \in \textnormal{dom }\varphi^\ast\\ \eta >0 }} \varphi^\ast(\eta x^\ast)=+\infty;$$ 
		\item Since $\varphi$ is epi-pointed, Moreau's Theorem (see Theorem \ref{theorem0}) 
tells us that $\varphi- x^\ast$ is weakly inf-compact for every $x^\ast \in \textnormal{Int}_{\tau(X^{\ast},X)}(\textnormal{dom } \varphi^\ast) 
	=K^{\circ}_s=\textnormal{dom } \varphi^\ast$, therefore $(\partial \varphi)^{-1} (x^*) \neq \emptyset$  for every $x^\ast \in \textnormal{dom } \varphi^\ast$.
	\end{enumerate} 
	We derive that $\varphi \in \mathcal{E}(K^\circ)$. By the arguments in (c) the function $\varphi$ is also weakly inf-compact. Further, since by its definition $h\geq 0$, we have $\varphi^*\geq 0$,  so the proof is complete.
\end{proof} 

         The third lemma studies the situation of another set constructed from a finite-valued lower $w^*$-semicontinuous convex function on $X^*$.  
		
\begin{lemma}\label{Lemmanuevo2}
Consider $\alpha ,\beta \in [-\infty,\infty]$.	Let $g: X^\ast \to \mathbb{R}$ be a weak$^\ast$ lower semicontinuous convex  function such that   
	\begin{align*}
	0\in \textnormal{lev}_g^{< }(\alpha,\beta):=&\{ x^\ast : \alpha <g(x^\ast) < \beta		\}, \text{ and}\\
		\partial g(x^\ast) \neq \emptyset, \; &\forall x^\ast \in  \textnormal{lev}_g^{< }(\alpha,\beta).
	\end{align*} 
	Then, there exists a  weak$^\ast$-neighborhood of zero $U \subseteq X^*$ such that  $$\mathcal{E}(U \cap \textnormal{lev}_g^{< }(\alpha,\beta)) \neq \emptyset;$$ 
	in fact there exists $\varphi \in \mathcal{E}(U \cap \textnormal{lev}_g^{< }(\alpha,\beta)) $  such that $\varphi^*$ is bounded from below on $X^*$
\end{lemma}
\begin{proof}
	Let us consider $U$ as a weak$^\ast$-closed convex balanced weak$^\ast$-neighborhood of zero such that $g(x^\ast) >\alpha$ for every $x^\ast \in U$. By Lemma \ref{Key:lemma} there exists a proper weakly inf-compact convex function $\varphi_1 \in \mathcal{E}(U)$ such that $\varphi_1^\ast$ is bounded from below. Now, define the function $\tilde{g}$ on $X^*$ by $\tilde{g}(x^\ast):=\max\{ g(x^\ast),\alpha \}$ for all $x^*\in X^*$.  Then, it is straightforward that $\tilde{g}$ is a  weak$^\ast$-lower semicontinuous convex  function, $0\in \textnormal{lev}_{\tilde{g}}^{< }(\beta)$ and for all $u^\ast \in  \textnormal{lev}_{\tilde{g}}^{< }(\beta)$
	\begin{align*}
	\partial \tilde{g} (u^\ast) \supseteq \left\{ \begin{array}{cl}
  	\partial g (u^\ast),  & \text{ if } g(u^\ast) >\alpha ,\\
  	0, & \text{ if } g(u^\ast) \leq \alpha. 
	\end{array}\right.
	\end{align*} 
Consequently by Lemma \ref{Lemmanuevo} there exists a proper lower semicontinuous convex function  
$\varphi_2 \in 	\mathcal{E}(\textnormal{lev}_{\tilde{g}}^{< }(\beta))$ and $\varphi_2$ is bounded from below on $X$ since 
	$(\partial\varphi_2)^{-1}(0)\neq\emptyset$  by $(i)$ and $(iv)$ in 
	Definition \ref{definition3.1}. {Moreover, $\varphi^*_2$  is bounded from below.}
 The convex function $\varphi :=\varphi_1 \square  \varphi_2$ is then proper lower semicontinuous 
(see \cite[\S4.e, p. 23]{Moreau}). Let us check that $\varphi \in \mathcal{E}(U \cap \textnormal{lev}_g^{< }(\alpha,\beta))$. Indeed
\begin{enumerate}[label={(\roman*)},ref={(\roman*)}]
		\item   $\varphi^\ast= \varphi_1^\ast + \varphi_2^\ast$ (see, e.g., \cite{Laurent,Moreau}), consequently $\varphi^\ast$ is proper and $\varphi^\ast(0)=0$.  
		\item Since $	\textnormal{dom}\,\varphi^\ast =  \textnormal{dom}\,\varphi_1 ^\ast\cap 
		\textnormal{dom}\,\varphi_2 ^\ast$ and $  U \cap \textnormal{lev}_g^{< }(\alpha,\beta)=   U \cap \textnormal{lev}_{\tilde{g}}^{< }(\beta)$ we have that 
		\begin{align*}
	  ( U \cap \textnormal{lev}_g^{< }(\alpha,\beta))_\infty &= 
	  (U \cap \textnormal{lev}_{\tilde{g}}^{< }(\beta))_\infty 
	  \subseteq U_\infty \cap (\textnormal{lev}_{\tilde{g}}^{< }(\beta))_\infty \\
	 & \subseteq \textnormal{dom}\,\varphi_1^\ast \cap \textnormal{dom}\,\varphi_2^\ast 
	=  \textnormal{dom}\,\varphi ^\ast
	=\textnormal{dom}\,\varphi_1 ^\ast\cap \textnormal{dom}\,\varphi_2^\ast   \\
	&  \subseteq  U \cap  \textnormal{lev}_{\tilde{g}}^{< }(\beta)= U \cap \textnormal{lev}_g^{< }(\alpha,\beta).
		\end{align*}     
		\item  {If $x^\ast \in \left( U \cap \textnormal{lev}_g^{< }(\alpha,\beta)\right)  \backslash ( U \cap \textnormal{lev}_g^{< }(\alpha,\beta))_\infty$ we have 
		\begin{align}\label{equation01}
		x^\ast \in \big(U \backslash U_\infty\big) \bigcup \big(\,\textnormal{lev}_g^{< }(\alpha,\beta) \backslash( \textnormal{lev}_g^{< }(\alpha,\beta))_\infty \,\big).
		\end{align}
	Considering $\lambda:=\min\{\lambda_1,\lambda_2 \}$, where $\lambda_1:=\sup\{ \nu >0 : \nu x^\ast \in U \}$ and $\lambda_2:=\sup\{ \nu >0 :\nu x^\ast \in \textnormal{lev}_g^{< }(\alpha,\beta)   \}$, it follows from  \eqref{equation01} that $\lambda \in \mathbb{R}$. Thus,
	 \begin{align}
	 \sup\limits_{ \substack{ \eta >0,\\ \; \eta x^\ast \in \textnormal{dom } \varphi^\ast}} \{ \varphi^\ast (\eta x^\ast) \} &\geq  \lim\limits_{  \eta \to \lambda^+} \{ \varphi_1^\ast (\eta x^\ast)+\varphi_2^\ast (\eta x^\ast) \},
	 	\end{align}
	 	so recalling that $0\in \textnormal{dom } \varphi_1^\ast \cap \textnormal{dom } \varphi_2^{\ast}$  
		and $\varphi_1^\ast$ and $\varphi_2^{\ast}$ are bounded from below, we obtain 
	 	$$
	 	 \sup\limits_{ \substack{ \eta >0,\\ \; \eta x^\ast \in \textnormal{dom } \varphi^\ast}}
		 \varphi^\ast (\eta x^\ast) =+\infty.
	 	$$}
		\item Take  $x^\ast \in \textnormal{dom} \varphi^\ast $. Since $\left(  \partial \varphi_1 \right)^{-1}(x^\ast)\neq \emptyset$ and $\left(  \partial \varphi_2 \right)^{-1}(x^\ast)\neq \emptyset$ 
by ($iv$) in Definition \ref{definition3.1} 	
		(because $\textnormal{dom} \varphi^\ast =\textnormal{dom} \varphi_1^\ast\cap \textnormal{dom} \varphi_2^\ast$), we have that   
	$$ 
	\partial \varphi^{\ast}(x^{\ast})\supset 
	 \partial\varphi_1^{\ast}(x^{\ast})+\partial\varphi_2^{\ast}(x^{\ast})\neq\emptyset,
	$$ 
	so 	$\left(  \partial \varphi\right)^{-1}(x^\ast)\neq \emptyset$ by \eqref{eq-SubdInvConj}.
\end{enumerate}
\end{proof}

We recall that the \emph{algebraic interior} $A^i$ of a set $A$ in $X$ is defined by
$$A^{i}:=\{ 		a \in A : \forall x\in X, \; \exists \delta >0,\; \forall \lambda \in [0,\delta], \; a + \lambda x \in A		\}.$$
\begin{theorem}\label{theoremconic}
	Let $\alpha,\beta\in [-\infty,+\infty]$ and  $f: X\to \mathbb{R}$ be a  function such that 
	$\textnormal{lev}_{f^\ast}^{< }(\alpha,\beta) $ has nonempty algebraic interior and  $$ \left( \partial f \right)^{-1}(x^\ast)\neq \emptyset, \; \forall  x^\ast \in \textnormal{lev}_{f^\ast}^{< }(\alpha,\beta).$$
	Then for every $x^\ast \in (\textnormal{lev}_{f^{\ast}}^{< }(\alpha,\beta))^{i}$
	and $\gamma \in \mathbb{R}$ the sets 
	\begin{align*}
	S_{f-x^*}(\gamma)&=\{  (x,\nu) \in X\times \mathbb{R} :	f(x) - \langle x^\ast , x \rangle \leq \nu \leq \gamma		\},\\
	\textnormal{lev}_{f-x^\ast}^{\leq}( \gamma)&=\{ 		x\in X : f(x) -\langle x^\ast , x \rangle \leq  \gamma		\}
	\end{align*}
	are relatively weakly compact.
\end{theorem}
\begin{proof}
	Take any $x^\ast \in  (\textnormal{lev}_{f^\ast}^{< }(\alpha,\beta))^{i}$. Considering 
	$f-x^*$ in place of $f$, we may and do suppose that $x^*=0$.  Then, the assumption of Lemma 
	\ref{Lemmanuevo2} are satisfied, and then there exists a weak$^\ast$-neighborhood $U$ of zero such that  $\mathcal{E}(U\cap \textnormal{lev}_{f^\ast}^{< }(\alpha,\beta)) \neq \emptyset$. Now since $0 \in  (\textnormal{lev}_{f^\ast}^{< }(\alpha,\beta))^{i}$  we have that $\mathbb{R}_{+}(U\cap \textnormal{lev}_{f^\ast}^{< }(\alpha,\beta)) =X^\ast$, then applying  Theorem \ref{mainresult}  we conclude that  $	S_{f-x^*}(\gamma)$ and 	$\textnormal{lev}_{f-x^\ast}^{\leq}( \gamma)$ are  relatively weakly compact.
\end{proof} 

    The result in the next corollary was first proved by B. Cascales, J. Orihuela and A. P\'erez  
\cite[Theorem 2]{CascOrihPere} for Banach spaces whole dual ball is weak$^{\ast}$ convex block compact, and it has been established for any Banach space by W. B. Moors \cite[Theorem 1]{MR3595216}.   
The corollary extends the result to complete locally convex spaces. 
\begin{corollary}
	Let $A$ and $B$ be nonempty bounded  closed convex sets of $X$ such that $0 \notin \textnormal{ cl} \left( A - B\right)$.
	If every $	x^\ast \in X^\ast$ with  \begin{align} \label{oneside}
	 \sup\limits_{x\in B} \langle x^\ast , x \rangle < \inf\limits_{x \in A} \langle x^\ast , x\rangle
	\end{align}	 attains its infimum on $A$ and  its supremum on $B$, then both $A$ and $B$ are weakly compact. 
	\end{corollary}
\begin{proof}
	Consider the set $C:=B-A$, and the function $f: = \delta_{C}$,  so $f^\ast = \sigma_{C}$. 
	
	We claim that  the set $\textnormal{lev}_{f^\ast}^{< }(0)=\textnormal{lev}_{f^\ast}^{< }(-\infty,0) $   has nonempty algebraic interior and $\left(\partial f	 \right)^{-1} (x^\ast) \neq \emptyset $ for all $x^\ast \in \textnormal{lev}_{f^\ast}^{< }(0)$.  On the one hand, to see that the set $\textnormal{lev}_{f^\ast}^{< }(0)$ has nonempty algebraic interior, by the assumption $0 \notin \textnormal{ cl} \left( A - B\right)$ choose some $x_0^\ast \in X^\ast$ such that $\sigma_{C}(x^\ast_0) <0$. Take any $x^{\ast}  \in X^\ast$. The boundedness of $C$ ensures that on $\mathbb{R}$ the convex function $h$ defined for all $t\in\mathbb{R}$ 
by  $h(t):=\sigma_C(x_0^{\ast} + t x^{\ast} )$ is finite, hence  continuous.  Then, by the inequality 
$h(0)<0$ there is a real $\delta>0$ such that 
	$\sigma_C(x^{\ast}_0+tx^{\ast})=h(t)<0$ for all $t\in[-\delta,\delta]$, which means that 
	$x^{\ast}_0+[-\delta,\delta]x^*\subset \textnormal{lev}_{f^\ast}^{< }(0)$. This says that 
	$x^{\ast}_0$ belongs to the algebraic interior of $\textnormal{lev}_{f^\ast}^{< }(0)$. 
	On the other hand, the set $\textnormal{lev}_{f^\ast}^{< }(0)$ is just the set of all points $x^\ast \in X^\ast$ such that  \eqref{oneside} holds, and  	we notice by \eqref{eq-SubdInvConj} and \eqref{eq-DefMRSubd} 
	that 
$$ 
  \left(\partial f	 \right)^{-1} (x^\ast)= \partial f^\ast(x^\ast)=\{ x\in C : \sigma_C(x^\ast)
	= \langle x^\ast , x\rangle  \}.
	$$ 
Hence, by the attainment assumption related to \eqref{oneside} the set $\left(\partial f	 \right)^{-1} (x^\ast)$ is non-empty for all $x^\ast \in \textnormal{lev}_{f^\ast}^{< }(0)$. The claim is then justified. 
	
	 Thus, Theorem  \ref{theoremconic} allows us to conclude that  for all $x^\ast \in (\textnormal{lev}_{f^\ast}^{< }(0))^{i}$ and $\gamma \in \mathbb{R}$  the set	$\textnormal{lev}_{f-x^\ast}^{\leq}( \gamma)$ is relatively weakly compact. In particular fixing $x_0^\ast \in (\textnormal{lev}_{f^\ast}^{< }(0))^{i}$ and putting
	$\gamma_0 := \sigma_{C} (-x_0^\ast)$ we have 
	that $C \subseteq \textnormal{lev}_{f-x_0^\ast}^{\leq}( \gamma_0)$. Therefore $A$ and $B$ are  relatively weakly compact.
	\end{proof}  

Now we show the mentioned extension of Theorem \ref{theorem1} using  Theorem \ref{mainresult}.  

\begin{theorem}\label{mainresult2} Let $U$ be a nonempty open set  in $X^{\ast}$ with respect to the  Mackey topology and 
	let $f: X \to \mathbb{R} \cup \{ +\infty \}$ be a function such that  $f-x^\ast$ attains its minimum for every $x^\ast \in U$. Then for every $\gamma  \in \mathbb{R}$ and every $x^\ast \in U$ the sets 
	\begin{align*}
	S_{f-x^*}(\gamma)&=\{  (x,\alpha) \in X\times \mathbb{R} :	f(x) - \langle x^\ast , x \rangle \leq \alpha \leq \gamma		\},\\
	\textnormal{lev}_{f-x^\ast}^{\leq}( \gamma)&=\{ 		x\in X : f(x) -\langle x^\ast , x \rangle \leq  \gamma		\}
	\end{align*}
	are relatively weakly compact.  In particular $f$ is an epi-pointed function.
\end{theorem}
\begin{proof}
Considering $x_0^\ast \in U$ and the transformation of the function $f \to f-x_0^\ast$, we may suppose that $0\in U$.  From the definition of the Mackey topology, there {exists a nonempty} convex circled weakly compact subset $K$ of $X$  such that  $K^{\circ} \subset U$, thus $f-x^\ast$ attains its minimum for every  $x^\ast \in K^{\circ}$,  
	which  is also equivalent to $\big( \partial f \big)^{-1}(x^\ast)\neq \emptyset$ for all  
$x^\ast \in K^{\circ}$. Noting also that $\mathbb{R}_+K^{\circ}=X^{\ast}$,  
by Lemma \ref{Key:lemma} we can take $\varphi \in \mathcal{E}(K^{\circ})$, then by Theorem \ref{mainresult} the sets  $S_f(\gamma)$	 and 	$\textnormal{lev}_f^{\leq}( \gamma)$  are  relatively weakly compact.   Finally, by Moreau's Theorem (see Theorem \ref{theorem0}), $f^*$ is Mackey-continuous at zero, and consequently $f$ is epi-pointed. 
\end{proof}

  From the above theorem we deduce the following extension  of  
	J. Orihuela's result in \cite[Theorem 2]{Orihuela} to  complete locally convex spaces.  
	The interest for such a result probably began with E. Jouini, W. Schachermayer and 
	N. Touzy \cite{Joui-Scha-Touz}. 
	
\begin{corollary}
	Let $D$ be a nonempty weakly compact subset of $X$ with $0\notin D$. If $A$ is a bounded subset of $X$ such that every $x^\ast \in X^\ast$ with $\inf_{x\in D} \langle x^\ast ,x \rangle >0$ attains its supremum on $A$, then $A$ is relatively weakly compact. 
\end{corollary}
\begin{proof}
	Consider the set $U:=\{ x^\ast \in X^\ast: 	\inf_{x\in D} \langle x^\ast ,x \rangle >0	\}$.  Due to the fact that $D$ is weakly compact, the set $U$ is an open set in $X^*$ with respect to the Mackey-topology. Then, applying Theorem \ref{mainresult2} with $f:= \delta_{A}$, we get that for every $\gamma  \in \mathbb{R}$ and every $x^\ast \in U$ the set  $\textnormal{lev}_{f-x^\ast}^{\leq}( \gamma)$ is relatively weakly compact. In particular, fixing a point $x^\ast_0\in U$ and putting  
	$\gamma_0 :=\sigma_{A}(-x^\ast_0)$, we have that   $A \subseteq \textnormal{lev}_{f-x_0^\ast}^{\leq}( \gamma_0)$, which implies that $A$ is relatively weakly compact.
\end{proof}
\begin{corollary}\label{corollary3.6}
	Let  $f:X\to \mathbb{R}\cup\{ +\infty\}$ be a proper lower semicontinuous  convex function and let $U$ be a nonempty  open set in $ X^\ast$ with respect to the Mackey topology. Then  the following statements are equivalent:
	\begin{enumerate}[label={(\alph*)},ref={(\alph*)}]
		\item\label{a} $f$ is epi-pointed and $U \subseteq \textnormal{dom} f^\ast$; 
		\item\label{b} $f-x^\ast$ is weakly inf-compact for all $x^\ast \in U$; 
		\item\label{c} $\big(\partial f \big)^{-1}(x^\ast) \neq \emptyset$, for all $x^\ast \in  U$.
	\end{enumerate}
\end{corollary}
\begin{proof}
	 Noting that $\ref{a}$ is equivalent to saying that $f^*$ is Mackey-continuous on the open set $U$, we see that the equivalence between \ref{a} and \ref{b}  follows from  Moreau's Theorem 
	(see Theorem \ref{theorem0}).

	On the other hand $\ref{b} \Rightarrow \ref{c}$ is straightforward. Finally,  
	$\ref{c} \Rightarrow \ref{b}$ is given by  Theorem \ref{mainresult2}. 
\end{proof}

\begin{example}
	{\rm 
	Consider any  nonempty weakly compact circled convex set $K\subseteq X$ such that $K^\circ \neq X^\ast$. The function $f:=\sigma_{K^\circ}$  satisfies the assumption of Theorem \ref{mainresult2} 
	with $U:=K^{\circ}_s$. However, it does not satisfy the assumption of Theorem \ref{theorem1}, because the effective domain of the subdifferential of $f^*=\delta_{K^\circ}$ is not the 
	whole space $X^*$.  } 
	\hfill $\square$
\end{example}

\begin{example}
	{\rm Consider a non-reflexive Banach space $X$. 
Define $f := \| \cdot\|$ on $X$ and note that its conjugate function is given by $f^\ast = \delta_{\mathbb{B}^\ast}$,  where $\mathbb{B}^\ast$ represents the closed unit ball in $X^\ast$ centered at zero.  Moreover,  for every 
	$x^\ast\in \mathbb{B}^\ast$ one has that $f^\ast (x^\ast) =  \langle x^\ast, 0 \rangle - f(0)$, that is, $f-x^*$ attains its minimum. Nevertheless, $\mathbb{B}=\{  x \in X : f(x) \leq 1 \}$ is not weakly compact.  This example shows that the assumption  $\mathcal{E}(\mathbb{B}^\ast)\neq \emptyset$ is crucial 
 in Theorem \ref{mainresult}.  } 
\hfill $\square$
\end{example} 

The last example shows that Theorem \ref{mainresult2} cannot be extended to the case when $U$  
therein is the open/closed unit ball in the dual space of a Banach space. However, we can establish the following characterization of semi-reflexivity in terms of the non-emptiness of  $\mathcal{E}(B^\circ)$  for some family of bounded sets $B$ in $X$. We recall that the strong topology  $\beta(X^{\ast},X)$ on $X^\ast$, denoted by  $\beta$ for short,  is the topology generated by the uniform convergence over bounded sets of $X$. The bidual of $X$, denoted by $X^{\ast \ast }$,  is the topological dual of $(X^\ast,\beta)$. The locally convex space $X$ is called \emph{semi-reflexive} if the canonical embedding (or evaluation mapping) $X \ni x \to  \langle \cdot , x\rangle  \in X^{\ast \ast}$ is onto. In contrast, $X$ is called \emph{reflexive} if  the canonical embedding is a homeomorphism (see  \cite{MR1741419} for more details).  It is worth mentioning that every semi-reflexive normed space is a reflexive Banach space (see \cite[Corollary 2, \S IV p. 145]{MR1741419}).

\begin{corollary}
	The space  $X$ is semi-reflexive if and only if $\mathcal{E}(B^\circ)\neq \emptyset$ for every bounded  circled convex  set $B$  in $X$. 
\end{corollary}
\begin{proof}
	By \cite[Theorem 5.5,  \S IV p. 144]{MR1741419}  we know that  a locally convex space $X$ is semi-reflexive if and only if every bounded subset of $X$ is relatively weakly compact. 
	
	   First, assume that $X$ is semi-reflexive and fix  any nonempty bounded circled convex set $B$ in $X$, so $B$ is relatively weakly compact. 
	Moreover	$\mathbb{R}_+B^{\circ}=X^*$. Indeed, taking any $x^\ast \in X^\ast$, by the boundedness of $B$ there exists some $\lambda >0$ such that 
\begin{align*}
\lambda B \subseteq V:=\{x^\ast\}^\circ =\{ x\in X:  \langle x^\ast , x \rangle \leq 1\}.
\end{align*}	
Then $\mathbb{R}_+B^{\circ}\supseteq \lambda^{-1}B^\circ \supseteq V^\circ \ni x^\ast$, 
which confirms by the arbitrariness of $x^\ast$ that $\mathbb{R}_+B^{\circ}=X^*$. Then, by Lemma \ref{Key:lemma} the set $\mathcal{E}(B^\circ)$ is  nonempty.   

   Conversely, take a nonempty bounded circled convex set $B$ of $X$ with $\mathcal{E}(B^{\circ})\neq\emptyset$. By Lemma \ref{Key:lemma} the set $B$ is relatively weakly compact. Since the boundedness is preserved under   circled convex hull, we deduce under the property of the corollary that every bounded subset of $X$ is relatively compact, so $X$ is semi-reflexive.  
\end{proof}

\bibliographystyle{plain}


\end{document}